\documentclass[a4paper,12pt,reqno]{amsart}
\usepackage{amsmath, amsthm, amssymb, amstext}

\usepackage[left=2.3cm,right=2.3cm,top=2.5cm,bottom=2.5cm]{geometry}

\usepackage{hyperref,xcolor}
\usepackage{tikz}
\usepackage[hyperpageref]{backref}
\hypersetup{pdfborder={0 0 0},colorlinks}

\usepackage{enumitem}
\allowdisplaybreaks
\newtheorem{theorem}{Theorem}[section]
\newtheorem{remark}{Remark}
\newtheorem{lemma}[theorem]{Lemma}
\newtheorem{proposition}[theorem]{Proposition}

\newtheorem{definition}[theorem]{Definition}

\numberwithin{equation}{section}

\title[A sharp eigenvalue theorem for mixed problems]
{A sharp eigenvalue theorem for mixed elliptic problems under mixed boundary conditions}

\author[G. Molica Bisci]{Giovanni Molica Bisci}
\thanks{$^{*}$Corresponding author: lovelesh.sharma@iiserpune.ac.in}
\address{Dipartimento di Matematica e Informatica, Universit\`a degli Studi di Perugia, Italy}
\email{giovanni.molicabisci@uniroma5.it}

\author[L. Sharma]{Lovelesh Sharma$^{*}$}
\address{Department of Mathematics, Indian Institute of Science Education and Research Pune, India}
\email{lovelesh.sharma@iiserpune.ac.in}


\begin{document}
\begin{abstract}
In this paper, we study a class of eigenvalue problems involving both local and nonlocal operators, namely the classical Laplacian and the fractional Laplacian, under mixed boundary conditions. More precisely, we consider the problem
\begin{equation}\label{1}
\left\{
\begin{aligned}
\mathcal{L}u &= \lambda f(u), \quad u>0 &&\text{in }\Omega,\\
u&=0 &&\text{in }U^c,\\
\mathcal{N}_s(u)&=0 &&\text{in }\mathcal{N},\\
\frac{\partial u}{\partial\nu}&=0
&&\text{on }\partial\Omega\cap\overline{\mathcal{N}},
\end{aligned}
\right.
\tag{$P_\lambda$}
\end{equation}
where
\(
U=\Omega\cup\mathcal{N}\cup
\bigl(\partial\Omega\cap\overline{\mathcal{N}}\bigr),
\)
\(\Omega\subseteq\mathbb{R}^n\) is a bounded open set with smooth boundary, \(\lambda>0\) is a real parameter, \(f\) is  continuous function with \(f(0)=0\), and
\[
\mathcal{L}=-\Delta+(-\Delta)^s,
\qquad s\in(0,1).
\]
We establish a characterization theorem for the existence of positive weak solutions to problem \eqref{1}. Motivated by the classical elliptic framework developed by
Molica Bisci and R\u{a}dulescu \cite{MolicaBisciRadulescu2017},
we establish a corresponding characterization result for mixed
local-nonlocal operators under mixed boundary conditions.
\end{abstract}
\maketitle

\section{Introduction}
We investigate the existence  of solutions to the following elliptic problem
\begin{equation*} \label{Main problem}
\left\{
\begin{aligned}
\mathcal{L}u &= \lambda f(u), \quad u > 0 \quad \text{in } \Omega, \\
u &= 0 \quad \text{in } U^c, \\
\mathcal{N}_s(u) &= 0 \quad \text{in } \mathcal{N}, \\
\frac{\partial u}{\partial \nu} &= 0 \quad \text{in } \partial \Omega \cap \overline{\mathcal{N}},
\end{aligned}
\right.
\tag{$P_\lambda$}
\end{equation*}
where  $U= (\Omega \cup {\mathcal{N}} \cup (\partial\Omega\cap\overline{\mathcal{N}}))$, $\Omega \subseteq \mathbb{R}^n$ is a non empty bounded open set with  smooth boundary $\partial\Omega$, say of class $C^1$ for  $n\geq 3$, $\mathcal{D}$, $\mathcal{N}$ are open subsets of $\mathbb{R}^n\setminus{\bar{\Omega }}$ such that $\overline{{\mathcal{D}} \cup {\mathcal{N}}}= \mathbb{R}^n\setminus{\Omega}$, $\mathcal{D} \cap {\mathcal{N}}=  \emptyset $ and $\Omega\cup \mathcal{N}$ is a bounded set with { sufficiently} smooth boundary.

The following pictures shows some possible configurations of the Dirichlet boundary 
$\mathcal{D}$ and the Neumann boundary $\mathcal{N}$
 with the set $\Omega$ which satisfies our assumptions, (see Fig. 1 and Fig 2). 
\begin{figure}[htp]
\centering
\begin{minipage}{0.48\textwidth}
\centering
\scalebox{0.55}{
\begin{tikzpicture}
  \definecolor{lightgrey}{RGB}{200, 200, 200}
  \fill[lightgrey] (0,0) circle (3cm);
  \fill[white] (0,0) circle (1.5cm);
  \draw[thick] (0,0) circle (3cm);
  \draw[thick] (0,0) circle (1.5cm);
  \node at (-2.2cm,1.5cm) {$\Omega$};
  \node at (0,0) {$\mathcal{N}$};
  \node at (3.4cm,0) {$\mathcal{D}$};
\end{tikzpicture}
}
\caption{Annular domain configuration.}
\label{fig:annular}
\end{minipage}
\hfill
\begin{minipage}{0.48\textwidth}
\centering
\scalebox{0.60}{
\begin{tikzpicture}
  \definecolor{lightgrey}{RGB}{200, 200, 200}
  \fill[lightgrey] (-2,0) circle (2cm);
  \draw[thick] (-2,0) circle (2cm);
  \node at (-2,0) {$\Omega$};
  \fill[lightgrey] (2,0) circle (2cm);
  \draw[thick] (2,0) circle (2cm);
  \node at (2,0) {$\mathcal{N}$};
  \node at (0,2.8) {$\mathcal{D}$};  
\end{tikzpicture}
}
\caption{Disconnected domain configuration.}
\label{fig:disconnected}
\end{minipage}
\end{figure}
Let $\lambda >0$ is a  real parameter, $\nu$ denotes the outward normal on $\partial\Omega\cap \overline{\mathcal{N}}$,
and
 $$\mathcal{L}=  -\Delta+(-\Delta)^{s},~ \text{for}~s \in  (0, 1).$$

The study of mixed operators of the type $\mathcal{L}$ as in the problem \eqref{1} is motivated by several applications where such kind of operators are naturally generated, including the theory of optimal searching, biomathematics, and animal forging for which we refer to \cite{dipierro2022non, dipierro2021description} and the references therein. In applied sciences, they are used for investigating the changes in physical phenomena that have both local and nonlocal effects. For instance, they are present in bi-modal power law distribution systems, see \cite{MR4225516}. Furthermore, they are present in models that are derived from the combination of two distinct scaled stochastic processes. We refer to  \cite{MR1640881} for a comprehensive explanation of this phenomenon. In recent years, there has been a significant amount of work investigating elliptic problems with mixed-type operator $\mathcal{L}$, which contain both local and nonlocal features. The analytical properties of elliptic and parabolic partial differential equations, as well as intergro-differential equations, relies significantly on  the spectrum of associated linearized problems.
In particular, the study of principal eigenvalues is essential in the investigation of non changing sign solutions to semi-linear problems as in \cite{biagi2022breziss} and in local bifurcation  phenomena as well as in stability analysis  (see \cite{berestycki2016definition}).

Over the past few decades, mixed Dirichlet-Neumann boundary value problems associated with elliptic equations of the form
\begin{equation*}
-\Delta u = g(x,u) \quad \text{in } \Omega,
\end{equation*}
have been extensively studied. In particular, several problems arising in physics and engineering, such as wave propagation, heat transfer, electrostatics, elasticity, and hydrodynamics, lead naturally to elliptic equations with mixed Dirichlet-Neumann boundary conditions. For motivation and related discussions, we refer to \cite{MR647124,MR1472351} and the references therein. A detailed analysis of the Laplace operator with mixed Dirichlet-Neumann boundary conditions can be found in \cite{MR2100910, MR2238024,MR3912863}.

Recently, Giacomoni et al. \cite{GiacomoniMukherjeeSharma} studied an eigenvalue problem involving mixed local and nonlocal operators under mixed Dirichlet-Neumann boundary conditions. They established the existence of the first eigenvalue and analyzed several of its qualitative properties, depending on the geometry of the sets $\mathcal{D}$ and $\mathcal{N}$. Moreover, they obtained bifurcation results both from zero and from infinity for asymptotically linear problems associated with this framework. 
In a related direction, Mukherjee et al. \cite{mukherjee2023nonlocal} investigated concave-convex elliptic problems within the same setting. Furthermore, singular critical problems \cite{mukherjeesharma} under this framework have also been studied, showing the applicability of these techniques to a wider class of nonlinear problems.

Motivated by recent developments on nonlocal elliptic equations with mixed boundary conditions, along with the variational techniques developed in \cite{GiacomoniMukherjeeSharma, mukherjee2023nonlocal} and the sharp eigenvalue results for fractional problems obtained in \cite{BisciRadulescuServadei2016}, we study a class of elliptic problems driven by mixed local and nonlocal operators under mixed boundary conditions. 

Our aim is to extend the existing results to this more general framework and to obtain a characterization theorem for the existence of solutions to problem \eqref{1}.
A key tool in our analysis is the validity of a Maximum Principle and regularity results for mixed local-nonlocal problems, recently established in \cite{GiacomoniMukherjeeSharma, mukherjee2023nonlocal}. We also make use of priori estimates for solutions of nonlocal equations, which play an important role in our approach.

Theorem~\ref{mainth} can be seen as an elliptic counterpart, in the setting of mixed local-nonlocal operators with mixed boundary conditions, of the abstract results obtained by Ricceri (see \cite{Ricceri2013} and \cite{Ricceri2014}). In those works, a general critical point theorem in Hilbert spaces is introduced and applied to obtain existence results. Related applications in the fractional framework can also be found in \cite{MolicaBisciRadulescu2017}.

Similar ideas have also been used in more involved settings, such as elliptic equations on fractal domains (see \cite{BisciRadulescu2015}). In addition, the regularity result proved in [\cite{MR4319042}, Theorem~3.7], is used to show that assumption $\mathbf{(h_1)}$ implies condition $\mathbf{(h_2)}$ in Theorem~\ref{mainth}.

Finally, we mention that our work is partly inspired by the classical paper of Brezis and Oswald \cite{BrezisOswald1986}, where semilinear problems driven by the Laplace operator are studied under the assumption that \(f(u)/u\) is decreasing on \((0,+\infty)\). In contrast, our approach relies only on a  condition, namely that \(F(\xi)/\xi^2\) is nondecreasing near the origin. This condition leads to a sublinear behavior of the nonlinearity; for example, for \(f(u)=u^p\), it corresponds to \(p\in(0,1]\), while hypothesis $\mathbf{(h_1)}$ excludes the linear case \(p=1\).

Moreover, unlike the classical setting, our existence result depends on the position of the parameter \(\lambda\) in a suitable interval and does not require any growth condition at infinity. This shows that, in the presence of mixed operators and mixed boundary conditions, local properties of the nonlinearity are sufficient to guarantee the existence of solutions.


\vspace{0.5cm}

\textbf{The article is organized as follows:} In Section~2, we introduce the functional framework appropriate for the mixed operator $\mathcal{L}$ under the mixed boundary conditions and recall some basic notions concerning the Hilbert spaces
\(\mathcal{X}^{1,2}_{\mathcal{D}}(U)\).
In Section~3 is devoted to the proof of Theorem~\ref{mainth}.

\section{Functional setting and Preliminaries}
In this section, we set our notations and formulated the functional framework for \eqref{1}, which are used throughout the paper.
For every $s\in (0,1)$, we recall the fractional Sobolev spaces
$${H^{s}(\mathbb{R}^n)} =  \Bigg\{ u \in L^{2}(\mathbb{R}^n):~~\frac{|u(x) - u(y)|}{|x - y|^{\frac{n}{2} + s}} \in L^{2}({\mathbb{R}^n}\times {\mathbb{R}^n)} \Bigg\} $$ which contain $H^1(\mathbb R^n)$. We assume that $\Omega \cup \mathcal N$ is bounded with a smooth boundary. 
The symbol $U$ is used throughout the article instead of $(\Omega \cup {\mathcal{N}} \cup (\partial\Omega\cap\overline{\mathcal{N}}))$ for sake of clarity.
We define the function space $\mathcal{X}^{1,2}_{\mathcal{D}}(U)$ as 
\begin{align*}
    \mathcal{X}^{1,2}_{\mathcal{D}}(U)  = \{u\in H^1(\mathbb{R}^n) : ~u|_{U} \in H^1_0(U) ~\text{and}~ u \equiv 0~ a.e. ~\text{in}~ {U^c}\}.
\end{align*}
Let us define 
 $$ \eta(u)^2 =  ||\nabla{u}||^2_{L^{2}(\Omega)}+ [u]^2_{s},$$
for $u\in\mathcal{X}^{1,2}_{\mathcal{D}}(U)$, where  $[u]_s$ is the Gagliardo seminorm of $u$ defined by 
 $$[u]^2_{s} = ~ \bigg(\int_{Q} \frac{|u(x)-u(y)|^{2}}{|x-y|^{n+2s}} \, dx dy \bigg)$$ and $Q= \mathbb R^{2n}\setminus (\Omega^c\times \Omega^c)$.
The following Poincar\'e type inequality can be established following the arguments of Proposition 2.4 in \cite{MR4065090} and taking advantage of partial Dirichlet boundary conditions in $U^c$.

\noindent
\begin{proposition}\label{Poin} 
(Poincar\'e type inequality) 
There exists a constant $C = C(\Omega, n, s)$
$> 0$ such that
\begin{align*}
\int_{\Omega} |u|^2 \, dx 
\leq C \bigg( \int_{\Omega} |\nabla u|^2 \, dx 
+ \int_{Q} \frac{|u(x) - u(y)|^2}{|x - y|^{n + 2s}} \, dx dy \bigg),
\end{align*}
for every $u \in \mathcal{X}^{1,2}_{\mathcal{D}}(U)$, i.e.,
\(
\|u\|^2_{L^2(\Omega)} \leq C \, \eta(u)^2.
\)
\end{proposition}
As a consequence of Proposition \ref{Poin}, $\eta(\cdot)$ forms a norm on $\mathcal{X}^{1,2}_{\mathcal{D}}(U)$ and $\mathcal{X}^{1,2}_{\mathcal{D}}(U)$ is a Hilbert space with the  inner product associated with $\eta(\cdot)$, defined for any $u,v\in \mathcal{X}^{1,2}_{\mathcal{D}}(U)$  by
$$\langle{ u},{ v}\rangle = \int_{\Omega} \nabla u. \nabla{v} \,dx + \int_{Q} {\dfrac{(u(x)-u(y)) (v(x)-v(y))}{|x-y|^{n+2s}}} ~ dx dy. $$
Consequently, we have the integration by-parts formula given in the following proposition. To proof, see \cite{GiacomoniMukherjeeSharma}.
 \begin{proposition}\label{P}
 For every $ u,v\in  C^\infty_0(U)$, it holds
\begin{align*}
    \int_{\Omega}v \mathcal{L} u \,dx  
    &= \int_{\Omega} \nabla u \cdot\nabla{v} \,dx +  \int_{Q} {\dfrac{(u(x)-u(y)) (v(x)-v(y))}{|x-y|^{n+2s}}} ~ dx dy\\
    &-  \int_{\partial \Omega\cap\overline{{\mathcal{N}}}} v {\frac{\partial u}{\partial \nu}}~ d{\sigma}-  \int_{{\mathcal{N}}} v {\mathcal{N}}_s u~ dx.
    \end{align*}
    where $\nu$ denotes the outward normal on $\partial\Omega$.
    \end{proposition}

Consequently to $\mathcal{X}^{1,2}_{\mathcal{D}}(U)\hookrightarrow H^1(\mathbb R^n)$ and Sobolev embeddings, we infer the following embedding result:

\begin{remark}\label{r2.5}
   For $U$ is bounded (since $\Omega\cup\mathcal{N}$ is bounded) with smooth boundary, then we have compact embedding
    $$\mathcal{X}^{1,2}_{\mathcal{D}}(U)\hookrightarrow \hookrightarrow L^q_{loc}(\mathbb{R}^n)$$
   for $q\in [1,2^*)$ and continuous embedding for $q\in[1, 2^*].$
\end{remark}
Let $f : [0,+\infty) \to [0,+\infty)$ be a continuous function satisfying $f(0)=0$. Suppose there exists a number $a > 0$ such that the function
\[
h(\xi) := \frac{F(\xi)}{\xi^{2}}, \quad \xi > 0,
\]
where $F(\xi) := \int_{0}^{\xi} f(t)\,dt$, is nonincreasing in the interval $(0,a]$.
Define
\[
\Theta := \sup\bigl\{ \eta > 0 : h \text{ is nonincreasing in } (0,\eta] \bigr\}
\in (0,+\infty].
\]
If $\Theta = +\infty$, then every positive solution 
$u \in \mathcal{X}^{1,2}_{\mathcal{D}}(U)$ is a solution of problem \eqref{1}.
If instead $\Theta < +\infty$, then any positive function 
$u \in \mathcal{X}^{1,2}_{\mathcal{D}}(U)$ is a weak solution of \eqref{1}.

Now, let
\begin{equation}\label{lambda}
\lambda_1(\mathcal{D})
=
\inf_{u \in \mathcal{X}^{1,2}_{\mathcal{D}}(U)\setminus\{0\}}
\frac{
\displaystyle
\int_{\Omega} |\nabla u|^2 \, dx
+
\int_{Q} \frac{|u(x)-u(y)|^2}{|x-y|^{n+2s}} \, dx \, dy
}{
\displaystyle
\int_{\Omega} |u|^2 \, dx
}.
\end{equation}
be the first positive eigenvalue of the linear problem

\begin{equation*}
\left\{
\begin{aligned}
\mathcal{L}u &= \lambda u, \quad u > 0 \quad \text{in } \Omega, \\
u &= 0 \quad \text{in } U^c, \\
\mathcal{N}_s(u) &= 0 \quad \text{in } \mathcal{N}, \\
\frac{\partial u}{\partial \nu} &= 0 \quad \text{in } \partial \Omega \cap \overline{\mathcal{N}},
\end{aligned}
\right.
\tag{$Q_\lambda$}
\end{equation*}
with corresponding eigenfunction $\varphi_s \in \mathcal{X}^{1,2}_{\mathcal{D}}(U)$, see \cite{GiacomoniMukherjeeSharma}.
\vspace{0.2cm}

We now define the notion of weak solution to \eqref{1}.
\begin{definition}\label{d1}
We say that a positive function $u\in \mathcal{X}^{1,2}_{\mathcal{D}}(U)$ is a weak solution to  \eqref{1} if 
\begin{equation}\label{dd2.2}
    \int_{\Omega} \nabla u\cdot\nabla \varphi \,dx +\int_{Q} \frac{(u(x)-u(y))(\varphi(x)-\varphi(y))}{|x-y|^{n+2s}} dxdy  = \lambda\int_{\Omega} f(u) \varphi\,dx,
\end{equation}
for all $\varphi \in \mathcal{X}^{1,2}_{\mathcal{D}}(U).$
\end{definition}
Now, the energy functional associated to the problem $\eqref{1}$
is
\[
 J_{\lambda}(u) = \frac{1}{2} \eta(u)^2
- \lambda \int_\Omega F(u(x)) dx,
\]
which is well defined for all $u\in~ \mathcal{X}^{1,2}_{\mathcal{D}}(U)$, where
\[
F(\xi) := \int_{0}^{\xi} f(t)\,dt, ~~\xi>0
\]

\vspace{0.2cm}
 Our main result is as follows.

\begin{theorem}\label{mainth}
Let $f : [0,+\infty) \to [0,+\infty)$ be a continuous function with $f(0)=0$.
Assume that there exists $a>0$ such that the function
\[
h(\xi) := \frac{F(\xi)}{\xi^{2}}, \quad \xi>0,
\]
where
\[
F(\xi) := \int_{0}^{\xi} f(t)\,dt,
\]
is nonincreasing on the interval $(0,a]$.
Then the following statements are equivalent:
\begin{itemize}
\item[$\mathbf{(h_1)}$]
The function $h$ is not constant on any interval $(0,\sigma]$ for any $\sigma>0$.

\item [$\mathbf{(h_2)}$] If $\lim_{\xi \to 0^{+}} h(\xi) < +\infty$, then for every $r>0$ there exists $\varepsilon_r>0$ such that, for all
\[
\lambda \in 
\Bigl(
\frac{\lambda_{1}(\mathcal{D})}{2\lim_{\xi \to 0^{+}} h(\xi)},
\;
\frac{\lambda_{1}(\mathcal{D})}{2\lim_{\xi \to 0^{+}} h(\xi)} + \varepsilon_r
\Bigr),
\]
problem \eqref{1} admits a weak solution
\[
u_\lambda \in \mathcal{X}^{1,2}_{\mathcal{D}}(U)
\quad \text{satisfying} \quad
\eta(u_\lambda) < r.
\]

\item[$\mathbf{(h_2^*)}$]
The function $f$ is subcritical, $\lim_{\xi \to 0^{+}} h(\xi) > 0$, and for every $r>0$ there exists an open interval $J_r \subset (0,+\infty)$ such that, for each $\lambda \in J_r$, problem \eqref{1} admits a weak solution
\[
u_\lambda \in \mathcal{X}^{1,2}_{\mathcal{D}}(U)
\quad \text{with} \quad
\eta(u_\lambda)< r.
\]
\end{itemize}

Moreover, if  $\lim_{\xi \to 0^{+}} h(\xi) = +\infty$, then the condition $\mathbf{(h_2)}$ reduces to the following form

\noindent

\medskip

\noindent
$\mathbf{(h_2')}$   for every $r>0$ there exists $\varepsilon_r>0$ such that, for all
\[
\lambda \in (0,\varepsilon_r),
\]
problem \eqref{1} admits a weak solution
\[
u_\lambda \in \mathcal{X}^{1,2}_{\mathcal{D}}(U)
\quad \text{with} \quad
\eta(u_\lambda)< r.
\]
\end{theorem}

\noindent

We first establish a basic result on the weak continuity of the nonlinear term, which will be used in the variational framework.

 \begin{lemma}\label{lem:weak-continuity-F}
Assume that $F:\mathbb{R}\to\mathbb{R}$ is continuous and satisfies the growth
condition
\begin{equation}\label{eq:growth}
|F(t)| \le \theta t^{2} + t_0 \quad \text{for all } t \in \mathbb{R},
\end{equation}
for some constants $\theta>0$ and $t_0> 0$.
Then the mapping
\[
u \longmapsto \int_{\Omega} F(u(x))\,dx
\]
is well defined and weakly continuous from
$\mathcal{X}^{1,2}_{\mathcal{D}}(U)$,
endowed with the weak topology, into $\mathbb{R}$.
Moreover, the functional
\[
J_{\lambda}(u) = \frac{1}{2} \eta(u)^2
- \lambda \int_\Omega F(u(x)) dx,
\]
well defined for all $u\in~ \mathcal{X}^{1,2}_{\mathcal{D}}(U)$
is weakly lower semicontinuous in
$\mathcal{X}^{1,2}_{\mathcal{D}}(U)$.
\end{lemma}
\begin{proof}
    Consequently, the map \[
u \longmapsto \int_{\Omega} F(u(x))\,dx
\]
is well defined. 
To prove weak continuity of this map,  
let \(\{u_j\}_{j\in\mathbb{N}}\) be a sequence in
\(\mathcal{X}^{1,2}_{\mathcal{D}}(U)\) such that
\[
u_j \rightharpoonup u
\quad \text{weakly in }
\mathcal{X}^{1,2}_{\mathcal{D}}(U).
\]
By Sobolev embedding result $\mathcal{X}^{1,2}_{\mathcal{D}}(U)\hookrightarrow  L^{\nu}_{loc}(\mathbb{R}^n)$, up to a subsequence,
\(\{u_j\}_{j\in\mathbb{N}}\) converges strongly to \(u\) in \(L^\nu_{loc}(\mathbb{R}^n)\) and
a.e. in \(\mathbb{R}^n\) as \(j \to +\infty\).
Moreover, there exists a function \(g \in L^\nu_{loc}(\mathbb{R}^n)\) such that
\begin{equation} \label{eq:3.5}
|u_j(x)| \leq g(x)
\quad \text{a.e. } x \in \mathbb{R}^n,
\end{equation}
for all \(j \in \mathbb{N}\) and any \(\nu \in [1,2^*_s)\).

By the continuity of \(F\) and growth condition \eqref{eq:growth}, it follows that
\[
F(u_j(x)) \to F(u(x))
\quad \text{a.e. } x \in \Omega
\]
and
\[
|F(u_j(x))|
\leq \theta u_j(x)^2 + t
\leq \theta g(x)^2 + t
\in L^1(\Omega),
\]
for almost every \(x \in \Omega\) and all \(j \in \mathbb{N}\).
Hence, by the Lebesgue Dominated Convergence Theorem,
\[
\int_{\Omega} F(u_j(x))\,dx
\to
\int_{\Omega} F(u(x))\,dx
\quad \text{as } j \to \infty.
\]
Therefore, the mapping
\[
w \mapsto \int_{\Omega} F(u(x))\,dx
\]
is continuous from
\(\mathcal{X}^{1,2}_{\mathcal{D}}(U)\),
endowed with the weak topology, into \(\mathbb{R}\).

On the other hand, the mapping
\[
u \mapsto
\eta(u)^2
\]
is convex; hence it is weakly lower semicontinuous.
Moreover,  functional \(J_\lambda\) is of class
\(C^1\) and weakly lower semicontinuous in
\(\mathcal{X}^{1,2}_{\mathcal{D}}(U)\).
\end{proof}
\vspace{0.2cm}
\section{Proof of Theorem~\ref{mainth}}

In this section, in order to prove of Theorem \ref{mainth}, first we establish the equivalence relation between conditions $\mathbf{(h_1)}$ and $\mathbf{(h_2)}$. To this end, we prove separately that
\[
\mathbf{(h_1)} \Rightarrow \mathbf{(h_2)}
\quad \text{and} \quad
\mathbf{(h_2)} \Rightarrow \mathbf{(h_1)}.
\]
The argument is naturally divided into two parts.

\subsection{\texorpdfstring{Part I: Proof of $\mathbf{(h_1)} \Rightarrow \mathbf{(h_2)}$}{Part I: Proof of (h1) implies (h2)}}\label{sec:part1}

We consider two cases depending on the value of $\Theta$:
\begin{enumerate}
\item[(i)] $\Theta = +\infty$;
\item[(ii)] $\Theta < +\infty$.
\end{enumerate}
\medskip

\noindent
\textbf{Case (i):} $\Theta = +\infty$.  

Assume that the function \(h\) is nonincreasing on  \((0,+\infty)\). Define
\[
\lim_{\xi \to 0^+} h(\xi) = t_1 \in (0,+\infty],
\]
and
\[
\lim_{\xi \to +\infty} h(\xi) = t_2.
\]
Clearly, \(t_2 \geq 0\) and, since assumption $\mathbf{(h_1)}$ holds, we have \(t_1 > t_2\). Hence,
\[
I := \left[ \frac{\lambda_{1}(\mathcal{D})}{2t_1}, \frac{\lambda_{1}(\mathcal{D})}{2t_2} \right] \neq \emptyset,
\]
where \(\lambda_{1}(\mathcal{D})\) is the first eigenvalue of \(\mathcal{L}\).

To this end, we first extend the function \(f\) to the whole real line by setting
\(f(t)=0\) for all \(t \in (-\infty,0)\).
Fix \(\lambda \in I\) and define the energy functional
\[
J_{\lambda}(u)
= \frac{1}{2} \eta(u)^2
- \lambda \int_{\Omega} F(u(x)) \, dx, ~~\forall~u \in \mathcal{X}^{1,2}_{\mathcal{D}}(U).
\]
The nontrivial critical points of \(J_\lambda\) coincide with the nontrivial weak solutions of problem \eqref{1}. Moreover, by using the maximum principle, see \cite{GiacomoniMukherjeeSharma}, if
\(u \in \mathcal{X}^{1,2}_{\mathcal{D}}(U)\) is a weak solution of \eqref{1}, 
is strictly positive a.e. in $U$ and solves problem \eqref{1}.

Since \(h\) is nonincreasing on \((0,+\infty)\), it follows that \(f\) has sublinear growth at \(+\infty\).
Indeed, by the definition of \(t_2\), there exist constants
\(\theta > t_2\) and \(t > 0\) such that
\begin{equation} \label{eq:3.2}
F(\xi) \leq \theta \xi^2 + t,~~ \forall~~\xi \in \mathbb{R}
\end{equation}
Using \eqref{eq:3.2} and the inequality
\(\xi f(\xi) \leq 2F(\xi)\) for all \(\xi \in \mathbb{R}\)
(which follows from the monotonicity of \(h\) and the fact that
\(f \equiv 0\) on \((-\infty,0)\)),
we obtain
\[
\xi f(\xi) \leq 2\theta \xi^2 + 2t,
\quad \forall~ \xi \geq 0.
\]
Hence,
\begin{equation} \label{eq:3.3}
f(\xi) \leq 2\theta |\xi| + \frac{2t}{|\xi|},
\quad \forall~ \xi \in \mathbb{R} \setminus \{0\}.
\end{equation}
Fixing \(\xi_0>0\), inequality \eqref{eq:3.3} yields
\[
f(\xi) \leq 2\theta |\xi| + \frac{2t}{\xi_0},
\quad \forall~ |\xi| \geq \xi_0.
\]
Therefore,
\[
f(\xi) \leq 2\theta |\xi| + \gamma,
\quad \forall~ \xi \in \mathbb{R},
\]
where
\[
\gamma := \max_{|\xi|\leq \xi_0} f(\xi) + \frac{2t}{\xi_0}.
\]
Since \(\lambda \in I\), we have
\[
\lambda < \frac{\lambda_{1}(\mathcal{D})}{2t_2}.
\]
Thus, we can choose \(\theta \in (t_1,t_2)\) such that
\begin{equation} \label{eq:3.4}
\frac{\lambda_{1}(\mathcal{D})}{2t_1}
< \lambda <
\frac{\lambda_{1}(\mathcal{D})}{2\theta}.
\end{equation}
Furthermore, from Lemma \ref{lem:weak-continuity-F}, we know that the functional $J_{\lambda}$ is lower semicontinuous in  $\mathcal{X}^{1,2}_{\mathcal{D}}(U).$
By \eqref{eq:3.2}, there exists a constant \(\alpha>0\), independent of \(\lambda\),
for instance \(\alpha :=  \frac{ t\lambda_{1}(\mathcal{D})}{2\theta}\), such that
\[
\lambda \int_{\Omega} F(u(x))\,dx
\leq
\lambda \theta \int_{\Omega} u^2(x)\,dx + \alpha
\leq
\frac{\lambda \theta}{\lambda_{1}(\mathcal{D})}
\eta(u)^2
+ \alpha.
\]
Therefore,
\begin{equation} \label{eq:3.6}
J_\lambda(u)
\geq
\left(\frac{1}{2}
- \frac{\lambda \theta}{\lambda_{1}(\mathcal{D})}\right)
\eta(u)^2
- \alpha, ~~~\forall~u \in \mathcal{X}^{1,2}_{\mathcal{D}}(U).
\end{equation}
In view of \eqref{eq:3.4}, term $\left(\frac{1}{2}-\frac{\lambda \theta}{\lambda_{1}(\mathcal{D})}\right)>0$, and from \eqref{eq:3.6}, we conclude that
\begin{equation} \label{eq:3.7}
J_\lambda(u) \to +\infty
\quad \text{as }
\eta(u) \to +\infty.
\end{equation}
Hence, the functional \(J_\lambda\) is coercive.

Now observe that \(\varphi_{s} \in \mathcal{X}^{1,2}_{\mathcal{D}}(U)\) and
\(\varphi_{s}\) is strictly positive a.e. in $U$.
Set
\[
\bar{\varphi}_s :=
\max_{x \in \overline{\Omega}} \varphi_s(x).
\]
Thanks to assumption $\mathbf{(h_1)}$, since $\varphi_s$ is not constant,
there exists a measurable set $\Omega_0\subset U$ with $|\Omega_0|>0$
such that $\varphi_s(x)<\bar\varphi_s$ for all
$x\in\Omega_0$. Hence,  for every $t>0$,
\[
h(t\varphi_s(x)) > h(t\bar{\varphi}_s),
\quad \forall~~ x \in \Omega_0,
\]
where \(\Omega_0 \subseteq U\) is a measurable subset with positive
Lebesgue measure.

Consequently, we obtain
\begin{align}
J_\lambda(t\varphi_s)
&= \frac{t^2}{2}
\eta(\varphi_s)^2
- \lambda \int_{\Omega}
h(t\varphi_s(x))(t\varphi_s(x))^2\,dx \notag \\
&<
\frac{t^2}{2}
\eta(\varphi_s)^2
- \lambda t^2 h(t\bar{\varphi}_s)
\int_{\Omega} (\varphi_s(x))^2\,dx \notag \\
&=
t^2 \eta(\varphi_s)^2
\left(
\frac{1}{2}
- \frac{\lambda}{\lambda_{1}(\mathcal{D})} h(t\bar{\varphi}_s)
\right).
\label{eq:3.8}
\end{align}
Now, setting \(\sigma := t^2 \bar{\varphi}_s\), we have
\[
\lim_{\sigma \to 0^+} h(\sigma)
=
\lim_{\sigma \to 0^+} \frac{F(\sigma)}{\sigma^2}
= t_1.
\]
If \(t_1 < +\infty\), then for every \(\varepsilon>0\) there exists
\(\delta_\varepsilon>0\) such that
\[
\frac{F(\sigma)}{\sigma^2} < t_1 + \varepsilon,
\quad \forall~ \sigma \in (0,\delta_\varepsilon].
\]
Since
\[
\frac{\lambda_{1}(\mathcal{D})}{2\lambda} < t_1,
\]
there exist \(\bar{\varepsilon}>0\) and \(\delta_{\bar{\varepsilon}}>0\) such that
\[
\frac{\lambda_{1}(\mathcal{D})}{2\lambda}
< t_1 - \bar{\varepsilon} < t_1,
\]
and
\[
\frac{F(\sigma)}{\sigma^2}
> t_1 - \bar{\varepsilon}
> \frac{\lambda_{1}(\mathcal{D})}{2\lambda},
\quad \forall~ \sigma \in (0,\delta_{\bar{\varepsilon}}].
\]
If \(t_1 = +\infty\), then
\[
\frac{F(\sigma)}{\sigma^2}
> \frac{\lambda_{1}(\mathcal{D})}{2\lambda}
\]
for \(\sigma\) sufficiently small.

Hence, in both cases, there exists \(\bar{t}>0\) such that
\[
h(\bar{t}\bar{\varphi}_s)
=
\frac{F(\bar{t}^2\bar{\varphi}_s)}{\bar{t}^2\bar{\varphi}_s}
>
\frac{\lambda_{1}(\mathcal{D})}{2\lambda}.
\]
Therefore,
\[
\inf_{u \in \mathcal{X}^{1,2}_{\mathcal{D}}(U)}
J_\lambda(u) < 0.
\]
Together with the weak lower semicontinuity of \(J_\lambda\)
and the coercivity property \eqref{eq:3.7}, this implies the existence of
a nontrivial global minimizer
\(u_\lambda \in \mathcal{X}^{1,2}_{\mathcal{D}}(U)\)
of the functional \(J_\lambda\).

As observed above, \(u_\lambda\) is a weak solution of problem \eqref{1}.
We now claim that
\begin{equation} \label{eq:3.9}
\lim_{\lambda \to \mu_0^+}
\eta(u_\lambda) = 0,
\end{equation}
where $\mu_0$ is defined by
\[
\mu_0 =
\begin{cases}
\dfrac{\lambda_{1}(\mathcal{D})}{2t_1}, & \text{if } t_1<+\infty,\\[6pt]
0, & \text{if } t_1=+\infty.
\end{cases}
\]
Now, to proving \eqref{eq:3.9}, let \(\{\lambda_j\}_{j\in\mathbb{N}} \subset
\left]\frac{\lambda_{1}(\mathcal{D})}{2t_1},
\frac{\lambda_{1}(\mathcal{D})}{2\theta}\right[\)
be a sequence such that
\[
\lim_{j\to\infty} \lambda_j = \frac{\lambda_{1}(\mathcal{D})}{2t_1}.
\]
It is easy to obseve that 
\(J_{\lambda_j}(u_{\lambda_j}) < 0\), for every \(j \in \mathbb{N}\).
Thus, from \eqref{eq:3.6} we obtain
\[
\eta(u_{\lambda_j})^2
<
\frac{\alpha}{\left(\frac{1}{2}
- \frac{\lambda_j\theta}{\lambda_{1}(\mathcal{D})}\right)}.
\]
Since
\[
\lim_{\lambda_j \to \mu_0^+}
\frac{\alpha}{\left(\frac{1}{2}
- \frac{\lambda_j\theta}{\lambda_{1}(\mathcal{D})}\right)}
=
\frac{\alpha}{\left(\frac{1}{2}
- \frac{\theta}{2t_1}\right)},
\]
the sequence \(\{u_{\lambda_j}\}\) is bounded in
\(\mathcal{X}^{1,2}_{\mathcal{D}}(U)\).

Up to a subsequence,
\(u_{\lambda_j} \rightharpoonup u_\infty\) weakly in
\(\mathcal{X}^{1,2}_{\mathcal{D}}(U)\), and
\[
u_{\lambda_j}
\to u_\infty
\quad \text{strongly in } L^l_{loc}(\mathbb{R}^n),
\quad \forall ~l \in [1,2^*_s).
\]

We claim that \(u_\infty = 0\).
Assume by contradiction that \(u_\infty \neq 0\).
For every \(\varphi \in
\mathcal{X}^{1,2}_{\mathcal{D}}(U)\), we have
\begin{align} \label{eq:3.10}
0 &=
 \int_{\Omega}
 \nabla u_{\lambda_j}\cdot \nabla \varphi  \,dx +\int_{Q}\frac{(u_{\lambda_j}(x)- u_{\lambda_j}(y))(\varphi(x)-\varphi(y))}{|x-y|^{n+2s}}\,dx\,dy
- \lambda_j \int_\Omega
f(u_{\lambda_j}(x))\varphi(x)\,dx .
\end{align}
Assume that \(t_1\) is finite (also the case \(t_1 = +\infty\) is similar). Taking into account inequality \eqref{eq:3.3}, and since \(u_j \rightharpoonup u_\infty\) weakly in \(\mathcal{X}^{1,2}_{\mathcal{D}}(U)\) and \(u_{\lambda_j}(x)) \to u_\infty\) strongly in \(L^1_{loc}(\mathbb{R}^n)\), passing to the limit in \eqref{eq:3.10}, we obtain
\begin{align} \label{eq:3.11}
0 &=
 \int_{\Omega}
 \nabla u_\infty\cdot \nabla \varphi  \,dx +\int_{Q}\frac{(u_{\infty}(x)- u_{\infty}(y))(\varphi(x)-\varphi(y))}{|x-y|^{n+2s}}\,dx\,dy
- \frac{\lambda_{1}(\mathcal{D})}{2t_1}
\int_\Omega f(u_\infty(x))\varphi(x)\,dx .
\end{align}
Therefore, \(u_{\infty}\) \(u_{\infty}\) is a weak solution of the problem \((P_{\mu_0})\) and again by the Maximum Principle, see \cite{GiacomoniMukherjeeSharma}, \(u_{\infty} > 0\) a.e. in \(U\).
Testing \eqref{eq:3.11} with \(\varphi = w_\infty\) and using
\(\xi f(\xi) \leq 2F(\xi)\), we deduce
\begin{align} \label{eq:3.12} 0 &= \eta(u_{\infty})^2 - \frac{\lambda_{1}(\mathcal{D})}{2t_1} \int_{\Omega} f(u_{\infty}(x))u_{\infty}(x)\,dx \notag \\ &\geq \eta(u_{\infty})^2 - \frac{\lambda_{1}(\mathcal{D})}{t_1} \int_{\Omega} F(u_{\infty}(x))\,dx \notag \\ &= \eta(u_{\infty})^2 - \frac{\lambda_{1}(\mathcal{D})}{t_1} \int_{\Omega} h(u_{\infty}(x))(u_{\infty}(x))^2 \,dx. \end{align}
Taking into account $\mathbf{(h_1)}$, relation \eqref{eq:3.12} yields \begin{align*} 0 &= \eta(u_{\infty})^2 - \frac{\lambda_{1}(\mathcal{D})}{t_1} \int_{\Omega} h(u_{\infty}(x))(u_{\infty}(x))^2 \,dx \\ &> \eta(u_{\infty})^2 - \lambda_{1}(\mathcal{D}) \int_{\Omega} (u_{\infty}(x))^2 \,dx, \end{align*} 
that implies \[
\lambda_{1}(\mathcal{D})
>
\frac{
\displaystyle \int_{\Omega} |\nabla u_{\infty}(x)|^{2}\,dx
\;+\;
\displaystyle \int_{Q} \frac{|u_{\infty}(x)-u_{\infty}(y)|^{2}}{|x-y|^{n+2s}}\,dx\,dy
}{
\displaystyle \int_{\Omega} |u_{\infty}(x)|^{2}\,dx
}.
\] in contradiction to the definition of \(\lambda_{1}(\mathcal{D})\). Therefore, we must have \(u_{\infty} = 0\).

Finally, choosing \(\varphi = u_{\lambda_j}\) in \eqref{eq:3.11}, we obtain
\[
\eta(u_{\lambda_j})^2
=
\lambda_j \int_\Omega
f(u_{\lambda_j}(x))u_{\lambda_j}(x)\,dx.
\]
Using \eqref{eq:3.3} and the strong convergence of
\((u_{\lambda_j})\) to zero in \(L^2_{loc}(\mathbb{R}^n)\),
we conclude that
\[
\lim_{j\to\infty}
\eta(u_{\lambda_j}) = 0,
\]
which proves \eqref{eq:3.9}.
\noindent

\textbf{Case (ii):} $\Theta < +\infty$.  
Suppose that $\Theta < +\infty$ and note that $h'(\Theta) = 0$. Define the function
\(h_0 : (0, +\infty) \to \mathbb{R}\) by
\[
h_0(\xi) := 
\begin{cases}
h(\xi) & \text{if } \xi \in (0, \Theta], \\
h(\Theta) & \text{if } \xi \in (\Theta, +\infty).
\end{cases}
\]
Since $h\in C^1((0,+\infty))$ and $h'(\Theta )=0$, the function $h_0$ coincides
with $h$ on $(0,\Theta]$ and is constant on $(\Theta,+\infty)$. Hence $h_0$
is continuously differentiable on $(0,+\infty)$.

Next, define
\[
F_0(\xi) := 
\begin{cases}
0 & \text{if } \xi \in (-\infty, 0], \\
h_0(\xi)\xi^2 & \text{if } \xi \in (0, +\infty) .
\end{cases}
\]
It follows that $F_0$ is a $C^1$ function and that $F_0(\xi) = F(\xi)$ for every $\xi \in (-\infty, \Theta]$. Moreover, $F_0$ satisfies condition $\mathbf{(h_1)}$, and the function
\[
\xi \mapsto \frac{F_0(\xi)}{\xi^2}
\]
is nonincreasing in $(0, +\infty)$.

We now consider the problem
\begin{equation} \label{Q_lambda_0}
\left\{
\begin{aligned}
\mathcal{L}u &= \lambda f_0(u), \quad u > 0 \quad \text{in } \Omega, \\
u &= 0 \quad \text{in } U^c, \\
\mathcal{N}_s(u) &= 0 \quad \text{in } \mathcal{N}, \\
\frac{\partial u}{\partial \nu} &= 0 \quad \text{in } \partial \Omega \cap \overline{\mathcal{N}},
\end{aligned}
\right.
\tag{$Q_\lambda^0$}
\end{equation}
\[
f_0(t) := F_0'(t) =
\begin{cases}
0 & \text{if } t \in (-\infty, 0], \\
f(t) & \text{if } t \in (0, \Theta], \\
2h_0(\Theta)t & \text{if } t \in (\Theta, +\infty).
\end{cases}
\]
Since case (i) finds a point 
$\mu_0$ where the desired property holds. Using continuity of the function involved, we know that property also holds in a small open interval 
$(\mu_0,\mu_0+\varepsilon_0), \varepsilon_0 > 0$. Then for any $r > 0$, there exists an open interval
\[
J := (\mu_0, \mu_0 + \varepsilon_0) 
\]
such that, for every $\lambda \in J$, problem \eqref{Q_lambda_0} admits a weak solution
$u_\lambda \in \mathcal{X}^{1,2}_{\mathcal{D}}(U)$ satisfying
\[
\eta(u_\lambda)< r.
\]
Furthermore, condition \eqref{eq:3.9} is satisfied.

Fix $q > \frac{n}{2s}\left(>\frac{n}{2}\right)$. Using [\cite{mukherjee2023nonlocal}, Lemma 3.4], there exists a  constant $C_q>0$ such that, for every $\mathcal{G} \in L^q(\Omega)$ and every solution
$u \in \mathcal{X}^{1,2}_{\mathcal{D}}(U)$ of the problem
\begin{equation} 
\left\{
\begin{aligned}
\mathcal{L}u &= \mathcal{G}(x),  \quad \text{in}~~ \Omega, \\
u &= 0 \quad \text{in } U^c, \\
\mathcal{N}_s(u) &= 0 \quad \text{in } \mathcal{N}, \\
\frac{\partial u}{\partial \nu} &= 0 \quad \text{in } \partial \Omega \cap \overline{\mathcal{N}},
\end{aligned}
\right.
\tag{$Q^{\mathcal{G}}_\lambda$}
\end{equation}
one has $u \in L^\infty(U)$ and
\begin{equation} \label{eq:3.14}
\|u\|_{L^{\infty}(\Omega)} \leq C_q \|\mathcal{G}\|_{L^q(\Omega)}.
\end{equation}
Since the function $f_0$ has sublinear growth at $\infty$ and satisfies $f_0(0) = 0$, there exists a positive constant $\vartheta>0$ and $C>0$ such that
\[
|f_0(t)| \le \vartheta |t| + C, \quad \forall ~~t \in \mathbb{R}.
\]
Choosing
\[
C = (\mu_0 + \varepsilon_0)^{-1} \frac{\Theta}{2 C_q |\Omega|^{1/q}},
\]
then we obtain the estimate
\begin{equation} \label{eq:3.15}
f_0(t) \leq \vartheta |t| + \frac{(\mu_0 + \epsilon_0)^{-1} \Theta}{2C_q |\Omega|^{1/q}}, ~~~\forall~~t \in \mathbb{R}.
\end{equation}
Let $u_\lambda \in \mathcal{X}^{1,2}_{\mathcal{D}}(U)$.
Then, for every
$\varphi \in \mathcal{X}^{1,2}_{\mathcal{D}}(U)$, we have
\[
\int_{\Omega}
 \nabla u_{\lambda}\cdot \nabla \varphi  \,dx +\int_{Q}\frac{(u_{\lambda}(x)- u_{\lambda}(y))(\varphi(x)-\varphi(y))}{|x-y|^{n+2s}}\,dx\,dy
= \lambda \int_{\Omega} f_0(u_\lambda(x)) \varphi(x)\,dx.
\]
Consequently, $u_\lambda \in \mathcal{X}^{1,2}_{\mathcal{D}}(U)$ is a weak solution of
\begin{equation} 
\left\{
\begin{aligned}
\mathcal{L}u &= \lambda f_0(u_\lambda(x)),  \quad \text{in}~~ \Omega, \\
u &= 0 \quad \text{in } U^c, \\
\mathcal{N}_s(u) &= 0 \quad \text{in } \mathcal{N}, \\
\frac{\partial u}{\partial \nu} &= 0 \quad \text{in } \partial \Omega \cap \overline{\mathcal{N}},
\end{aligned}
\right.
\end{equation}
By \eqref{eq:3.14} and \eqref{eq:3.15}, we obtain
\begin{equation} \label{eq:3.17}
\|u_\lambda\|_{L^{\infty}(\Omega)}
\leq \lambda \vartheta C_q \|u_\lambda\|_{L^{q}(\Omega)}
+ \lambda \lambda_0^{-1} \frac{\Theta}{2}
\leq \lambda_0 \vartheta C_q \|u_\lambda\|_{L^{q}(\Omega)} + \frac{\Theta}{2},
\end{equation}
for every $\lambda \in J$.

Fix $0<\mu<1$ such that $q\mu < 2^*$. We now prove that
\begin{equation} \label{eq:3.18}
\|u_\lambda\|_{L^{q}(\Omega)}
\leq \|u_\lambda\|_{L^{\infty}(\Omega)}^{1-\mu}
\|u_\lambda\|_{L^\mu q(\Omega)}^{\mu}.
\end{equation}
Indeed,
\[
\begin{aligned}
\|u_\lambda\|_q
&:= \left( \int_{\Omega} |u_\lambda(x)|^q \, dx \right)^{\frac{1}{q}} \\[1mm]
&= \left( \int_{\Omega}
|u_\lambda(x)|^{q(1-\mu)} |u_\lambda(x)|^{\frac{q}{\mu}} \, dx
\right)^{\frac{\mu}{q\mu}}
\leq \|u_\lambda\|_{L^{\infty}(\Omega)}^{1-\mu}
\|u_\lambda\|_{L^{\mu q}(\Omega)}^{\mu}.
\end{aligned}
\]
By \eqref{eq:3.18} and the Sobolev embedding
$\mathcal{X}^{1,2}_{\mathcal{D}}(U) \hookrightarrow L^{q\mu}_{loc}(\mathbb{R}^n)$, we deduce that
\begin{equation}\label{est}
\|u_\lambda\|_q
\leq \|u_\lambda\|_\infty^{1-\mu}
\|u_\lambda\|_{\mu q}^{\mu}
\leq C_{q\mu}
\|u_\lambda\|_{L^{\infty}(\Omega)}^{1-\mu}
\eta(u_\lambda)^{\mu},
    \end{equation}
for some positive constant $C_{q\mu}>0$.

Substituting \eqref{est}  into \eqref{eq:3.17}, we obtain
\begin{equation} \label{eq:3.19}
\|u_\lambda\|_{L^{\infty}(\Omega)}
\leq k \|u_\lambda\|_{L^{\infty}(\Omega)}^{1-\mu}
\eta(u_\lambda)^{\mu}
+ \frac{\Theta}{2},
\end{equation}
for some positive constant $k$.

From \eqref{eq:3.9} and \eqref{eq:3.19}, it follows that
\[
\lim_{\lambda \to \mu_0^+} \|u_\lambda\|_\infty
\leq \frac{\Theta}{2}.
\]
Consequently, there exists $0<\varepsilon_1<\varepsilon_0$ such that
$u_\lambda(x) \leq \Theta$ for almost every $x \in \Omega$ and for every
\[
\lambda \in J' := ]\mu_0, \mu_0 + \varepsilon_1[.
\]
In conclusion, for every $\lambda \in J'$, the function
$u_\lambda \in \mathcal{X}^{1,2}_{\mathcal{D}}(U)$ is a weak solution of problem
\eqref{1} and satisfies
\[
\eta(u_\lambda) < r,
\]
since $J' \subset J$. This completes the proof.

\subsection{\texorpdfstring{Part II: $\mathbf{(h_1)} \Leftarrow \mathbf{(h_2)}$}{Part II: (h1) is implied by (h2)}} \label{sec:part2}

We proceed by contradiction. Assume  that there exist  constants \(c_1, c_2>0\)  such that
\begin{equation} \label{eq:linF}
F(\xi) = c_2\xi^2~~~\forall~~\xi \in [0, c_1].
\end{equation}
 As a consequence,
\begin{equation} \label{eq:linf}
f(\xi) = 2c_2\xi, ~~~\forall~~\xi \in [0, c_1].
\end{equation}
Let \(\{a_j\}_{j \in \mathbb{N}} \subset (0, +\infty)\) be a sequence satisfying \(\lim_{j \to \infty} a_j = 0\). Then, for each \(j \in \mathbb{N}\), there exists \(\varepsilon_j > 0\) such that, for every
\[
\lambda \in J_j := ]\mu_0, \mu_0 + \varepsilon_j[,
\]
the problem
\begin{equation} \label{3}
\left\{
\begin{aligned}
\mathcal{L}u &= \lambda f(u),  \quad \text{in}~~ \Omega, \\
u &= 0 \quad \text{in } U^c, \\
\mathcal{N}_s(u) &= 0 \quad \text{in } \mathcal{N}, \\
\frac{\partial u}{\partial \nu} &= 0 \quad \text{in } \partial \Omega \cap \overline{\mathcal{N}},
\end{aligned} \tag{$P_\lambda^\#$}
\right.
\end{equation}
admits a positive weak solution \(u_{\lambda,j} \in \mathcal{X}^{1,2}_{\mathcal{D}}(U)\) such that
\(\eta(u_{\lambda,j}) < a_j\).

In particular, it follows that
\begin{equation} \label{eq:3.20}
\lim_{j \to \infty} \sup_{\lambda \in J_j} \eta(u_{\lambda,j}) = 0.
\end{equation}
Next, arguing as in the previous \textbf{(part I)}, we can find a positive constant \(k\), independent of both \(j\) and \(\lambda\), and a sufficiently small \(\mu\) such that
\begin{equation} \label{eq:3.21}
\|u_{\lambda,j}\|_{L^{\infty}(\Omega)} \leq k\|u_{\lambda,j}\|_{L^{\infty}(\Omega)}^{1-\mu}\eta(u_{\lambda,j})^{\mu} + \frac{c_1}{2},~~\forall~\lambda \in J_j, ~j \in \mathbb{N}.
\end{equation}
Combining \eqref{eq:3.20} with the above inequality, we deduce that
\begin{equation} \label{eq:3.22}
\lim_{j \to \infty} \sup_{\lambda \in J_j} \|u_{\lambda,j}\|_{{L^{\infty}(\Omega)}} \leq \frac{c_1}{2}.
\end{equation}
Therefore, there exists \(j_0 \in \mathbb{N}\) such that
\begin{equation} \label{eq:3.23}
\|u_{\lambda,j_0}\|_{{L^{\infty}(\Omega)}} \leq c_1, ~~\forall~\lambda \in J_{j_0}
\end{equation}
As a consequence, for every
\[
\lambda \in [\mu_0, \mu_0 + \varepsilon_{j_0}],
\]
the problem \eqref{3} admits a weak solution
\(u_{\lambda,j_0} \in \mathcal{X}^{1,2}_{\mathcal{D}}(U)\).
This leads to a contradiction, since the  problem \eqref{3} admits solutions only for countably many positive values of the parameter \(\lambda\).
Hence,  proof of Theorem \ref{mainth}
is completed now.
\qed

\section*{Declarations}

\noindent\textbf{Ethical Approval.}
Not applicable.

\medskip
\noindent\textbf{Competing interests.}
The authors declare that they have no competing interests.

\medskip
\noindent\textbf{Authors contributions.}
The authors contributed equally to this work.

\medskip
\noindent\textbf{Availability of data and materials.}
Data sharing is not applicable to this article as no new data were created or analyzed in this study.

\section*{Acknowledgement}
Lovelesh Sharma gratefully acknowledges the financial support provided
by the Institute Postdoctoral Fellowship at the Indian Institute of
Science Education and Research (IISER) Pune.

\end{document}